\documentclass[11pt,a4paper]{amsart}
\usepackage{amsmath,amsthm}
\usepackage[T1]{fontenc}
\usepackage{textcomp}
\usepackage{lmodern}

\usepackage{xspace}
\usepackage{enumitem}
\setenumerate[1]{label=\arabic{*}), ref=(\arabic{*})}
\setenumerate[2]{label=\roman{*}), ref=(\roman{*})}
\newtheorem{thm}{Theorem}[section]

\newtheorem{lem}[thm]{Lemma}
\newtheorem{prop}[thm]{Proposition}

\newtheorem{eks}{\sc Example}
\theoremstyle{definition}
\newtheorem{defn}[thm]{Definition}
\theoremstyle{remark}
\newtheorem{rem}{Remark}[section]
\numberwithin{equation}{section}
\newcommand{\xs}{x^\ast}
\newcommand{\Xs}{X^{\ast}}
\newcommand{\Xss}{{X^{\ast\ast}}}
\newcommand{\Xsss}{X^{\ast\ast\ast}}
\newcommand{\xss}{x^{\ast\ast}}
\newcommand{\Ys}{Y^{\ast}}
\newcommand{\ys}{y^\ast}
\newcommand{\yss}{y^{\ast\ast}}
\newcommand{\Yss}{Y^{\ast\ast}}
\newcommand{\Zs}{Z^\ast}
\newcommand{\zs}{z^\ast}

\newcommand{\RNP}{Radon-Nikod\'{y}m property\xspace}
\DeclareMathOperator{\conv}{conv}

\begin{document}

\title[Remarks on diameter 2 properties]%
{Remarks on diameter 2 properties}%

\author[T.~A. Abrahamsen]{Trond Abrahamsen}
\author[V.~Lima]{Vegard Lima}
\author[O.~Nygaard]{Olav Nygaard}

\address{Aalesund University College, Postboks 1517,
N-6025 {\AA}lesund, Norway.}
\email{Vegard.Lima@gmail.com}

\address{Department of Mathematics, Agder University, Servicebox 422,
4604 Kristiansand, Norway.}
\email{Trond.A.Abrahamsen@uia.no}
\urladdr{http://home.uia.no/trondaa/index.php3}
\email{Olav.Nygaard@uia.no}
\urladdr{http://home.uia.no/olavn/}

%
%
\begin{abstract}
  If $X$ is an infinite-dimensional uniform algebra, if $X$ has the
  Daugavet property or if $X$ is a proper $M$-embedded space, every
  relatively weakly open subset of the unit ball of the Banach space
  $X$ is known to have diameter 2, i.e., $X$ has the diameter 2
  property. We prove that in these three cases even every finite
  convex combination of relatively weakly open subsets of the unit
  ball have diameter 2. Further, we identify new examples of spaces
  with the diameter 2 property outside the formerly known cases; in
  particular we observe that forming $\ell_p$-sums of diameter 2
  spaces does not ruin diameter 2 structure.
\end{abstract}
\maketitle

\section{Introduction}

Let $X$ be a (real) Banach space and $B_X$ its unit ball.
By a slice of $B_X$
we mean a set of the type
$S(\xs, \varepsilon)=\{x\in B_X\::\xs(x)>1-\varepsilon\}$
where $\xs$ is in the unit sphere $S_{\Xs}$ of $\Xs$
and $\varepsilon > 0$.
A finite convex combination of slices of $B_X$
is then a set of the form
\begin{equation*}
  S=\sum_{i=1}^{n}\lambda_i S(\xs_i,\varepsilon_i),\:\:\lambda_i \ge
  0,\:\:\sum_{i=1}^{n}\lambda_i=1,
\end{equation*}
where $\xs_i \in S_{\Xs}$ and
$\varepsilon_i > 0$ for $i=1,2,\ldots,n$.

Let us consider the following three properties:
\begin{defn}\label{defn:diam2p}
  A Banach space $X$ has the
  \begin{enumerate}
  \item \emph{local diameter 2 property}
    if every slice of $B_X$ has diameter 2.
  \item \emph{diameter 2 property}
    if every non-empty relatively weakly open subset of $B_X$ has
    diameter 2.
  \item \emph{strong diameter 2 property}
    if every finite convex combination of slices
    of $B_X$ has diameter 2.
  \end{enumerate}
\end{defn}

Before going further, let us just remark that the following
implications hold true for the properties in
Definition \ref{defn:diam2p}: 
\begin{equation*}
  \text{strong diameter 2} \Rightarrow \text{diameter 2} \Rightarrow
  \text{local diameter 2}.
\end{equation*}
The second implication is clear since every slice of $B_X$ is a
non-empty relatively weakly open subset of $B_X$. The first
implication is a consequence of a lemma by Bourgain saying that every
non-empty relatively weakly open subset of $B_X$ contains a finite
convex combination of slices (see \cite[Lemma~II.1~p.~26]{GGMS}). Note
that a finite convex combination of slices need not be relatively
weakly open (it may be contained in $\delta B_X$, where $\delta<1$,
see \cite[Remark~IV.5~p.~48]{GGMS}, and recall that any non-empty
relatively weakly open subset of $B_X$ must intersect $S_X$ when $X$ is
infinite-dimensional.)  Note also that the strong diameter 2 property
implies that every non-empty finite convex combination of relatively
weakly open sets in $B_X$ has diameter 2.

To the best of our knowledge it is not known whether any of the
reverse implications hold.  Nor do we know of any general
characterizations of the diameter 2 properties.  The diameter 2
property as a term has been in use for some years now.  The idea of
studying its local and strong versions appears to be new.

It is an interesting exercise to show that the classical spaces $c_0$,
$C[0,1]$ and $L_1[0,1]$ have the diameter 2 properties.  Using
dentability it is clear that spaces with the \RNP cannot have the
local diameter 2 property.  It is also clear that spaces with the
point of continuity property cannot have the diameter 2 property.
Thus diameter 2 properties are at the opposite side of the spectrum
from the Radon-Nikod\'{y}m and point of continuity properties.  Note
that the predual $B$ of the James tree space has the point of
continuity property but lacks the \RNP (see \cite[Example~6.(1)]{EW}).
Since the point of continuity property is preserved by renormings, $B$
can not be renormed to have the diameter $2$ property.  It is an open
question (see e.g. \cite[p.~553]{BGLP}) whether every Banach space
failing the \RNP can be renormed to have the
local diameter $2$ property.

Section 2 is a survey of examples and results on diameter 2
properties. Particular emphasis has been put on checking which of the
diameter 2 properties that in fact are known to hold in each
case. Through this survey we will motivate the research we have done,
and the reader can easily see our main theorems in the light of known
results. 

Section 3 contains our perhaps most surprising result.
While in Section 2 spaces with $L$- and $M$-structure
dominate the landscape of diameter 2 spaces,
we show that the diameter 2 property is actually
preserved when taking any $\ell_p$-sum, $1\le p \le \infty$,
of diameter 2 spaces.
In particular, $c_0 \oplus_2 c_0$ has the diameter 2 property,
and we show that this example is not covered
by the results in Section~2.

In Section~4 we show that spaces with the Daugavet property,
infinite-dimensional uniform algebras, proper $M$-embedded spaces, and
biduals of proper $M$-embedded spaces all have the strong diameter 2
property.

We end the paper with some concluding remarks and questions.

We use standard Banach space terminology and notation.

\section{A survey of examples and results on diameter 2 properties}

We will now try to give an overview of \em general \em results on
diameter 2 properties for Banach spaces.

Recall that a Banach space $X$ has the Daugavet property if 
for every rank one operator $T: X \to X$ the equation
$\|I+T\| =  1 + \|T\|$ holds where $I$ is the identity operator on $X$.
One can show that the Daugavet property is equivalent to each of the
following two statements (see \cite{W} or \cite{SHV}):
\begin{itemize}
\item[1.] For every slice $S=S(\xs_0,\varepsilon_0)$ of $B_X$, every
  $x_0 \in S_X$, and every $\varepsilon > 0$
  there exists a point $x \in S$ such that
  $\|x + x_0\| \ge 2 - \varepsilon$.
\item[2.] For every $\varepsilon > 0$, every $x \in
  S_X$, $B_X = \overline{\conv}\Delta_\varepsilon(x)$ where
  $\Delta_\varepsilon(x)=\{y \in B_X: \|y-x\| \ge 2-\varepsilon\}$.
\end{itemize}
The first general results date some ten years back. From \cite{SHV}
and \cite{NW} we have:

\begin{thm}\label{teo1}
  Let $X$ be a space with the Daugavet property or an
  infinite-dimensional uniform algebra. Then $X$ has the diameter 2
  property.
\end{thm}

In Section \ref{strongsection} we will prove that spaces with the
Daugavet property and infinite-dimensional uniform algebras in fact
have the strong diameter 2 property. 

Note that the Daugavet property can be weakened as in the proposition
below (see also Problem (7) in \cite{W}) and still imply the local
diameter 2 property. 

\begin{prop}
  Let $X$ be a Banach space such that $x \in
  \overline{\conv}\Delta_\varepsilon(x)$
  for every $x \in S_X$ and $\varepsilon > 0$.
  Then $X$ has the local diameter 2 property. 
\end{prop}

\begin{proof}
  Let $\xs \in S_{\Xs}$ and $\varepsilon > 0$. Find $x \in B_X$ such
  that $\xs(x) \ge 1 -\varepsilon/2$ and a convex combination
  $\sum_{i=1}^n \lambda_i y_i$ where $y_i \in \Delta_\varepsilon(x)$
  such that $\|\sum_{i=1}^n \lambda_i y_i - x\| \le \varepsilon/2$. In
  particular $\xs(\sum_{i=1}^n \lambda_i y_i) \ge 1-\varepsilon$. If
  $\xs(y_i) < 1 - \varepsilon$ for every $i=1,2,\ldots,n$, then
  \begin{equation*}
    \xs(\sum_{i=1}^n \lambda_i y_i) < 1 - \varepsilon.
  \end{equation*}
  This is a
  contradiction, so there is at least one $i$ such that $\|x - y_{i}\|
  \ge 2 - \varepsilon$ with $x$ and $y_{i}$ in the slice
  $S(\xs,\varepsilon) = \{z \in B_X: \xs(z) \ge 1-\varepsilon\}$.
\end{proof}

Recall that a closed subspace $X$ of a Banach space $Y$ is an
$L$-summand (u-summand) in $Y$ if there is a subspace $Z$, the
$L$-complement (u-complement) of $X$, so that $X \oplus Z =Y$ and if
$x \in X, z \in Z$ then $\|x+z\|=\|x\|+\|z\|$ ($\|x+z\|=\|x-z\|$).  If
the annihilator $X^{\perp}$ of $X$ is an $L$-summand (u-summand) in
$\Ys$, then $X$ is said to be an $M$-ideal (u-ideal) in $Y$. Banach
spaces which are $M$-ideals in their biduals are called
$M$-embedded. Strict u-ideals are u-ideals in their biduals where the
u-complement of $X^\perp$ is $\Xs$. 

In \cite{BLPR} a very important observation is the following:

\begin{thm}\label{teo2} If $\Xs$ is a proper $L$-summand in $\Xsss$,
  then $X$ has the diameter 2 property. In particular, proper
  $M$-embedded spaces have the diameter 2 property.
\end{thm}

The authors of \cite{BLPR} prove that real $JB^\ast$-triples over
non-reflexive Banach spaces get the diameter 2 property from Theorem
\ref{teo2}. 
The role of $M$-structure is taken further in \cite{LP}:

\begin{thm}\label{teo3} 
Assume $Y$ is a proper $M$-ideal in $X$ and put $\Xs = Z \oplus_1
Y^{\perp}$. If $Z$ is 1-norming for $X$, then both $X$ and $Y$ have
the diameter 2 property. In particular, if $Y$ is a proper
$M$-embedded space, then both $Y$ and $\Yss$ have the diameter 2
property. In fact, in the latter case, every subspace $Z$ of $\Yss$
which contains $Y$ has the diameter 2 property.
\end{thm}

In \cite{BGRP} the centralizer is introduced to the study of the
diameter 2 property. The centralizer of $X$ (we write $Z(X)$), is the
set of those multipliers $T$ on $X$ such that there exists a
multiplier $S$ on $X$ satisfying $a_S(p)=\overline{a_T(p)}$ for every
extreme point $p$ of $B_{\Xs}$. Recall that a multiplier on $X$ is a
bounded linear operator $T$ on $X$ such that every extreme point of
$B_{\Xs}$ becomes an eigenvector for the adjoint $T^\ast$ of $T$. Thus
given a multiplier $T$ on $X$ and an extreme point $p$ of $B_{\Xs}$,
there exists a unique number $a_T(p)$ satisfying $p \circ T
=a_T(p)p$. 

Let $X$ be a Banach space. Using the notation of \cite{ABG} consider
the increasing sequence of even duals
\begin{equation*}
  X \subseteq \Xss \subseteq X^{(4} \cdots \subseteq X^{(2n} \subseteq
  \cdots,
\end{equation*}
and define $X^{(\infty}$ as the completion of the normed space
$\cup_{n=0}^\infty X^{(2n}$. 

In \cite{BGRP} the following result is proved:

\begin{thm}\label{teo4}
  If $Z(X^{(\infty})$ of a Banach space $X$ is
  infinite-dimensional, then $X$ has the diameter 2 property. In fact,
  if $X$ fails the diameter 2 property, then $\sup_n\dim
  Z(X^{(2n})<\infty$.
\end{thm}

Theorem~\ref{teo4} includes Theorem~\ref{teo2} and a lot of other
cases (see \cite[Proposition 3.3]{ABG}). As Theorem~\ref{teo4}
indicates, $Z(X^{(\infty})$ being infinite-dimensional more than
suffices to assure that $X$ has the diameter 2 property. In \cite{ABG}
this becomes very clear; we state the main result:

\begin{thm}\label{teo5}
  If $\dim Z(X^{(\infty})=\infty$, then the
  completed $n$-fold symmetric projective tensor product of $X$ has
  the diameter 2 property.
\end{thm}

The results involving the centralizer are strong and powerful, but
they do not contain the $L_1$-cases. We have now stated the most
fundamental theorems of sufficiency for diameter 2 properties. Let us
see how diameter 2 properties are transferred to some basic
structures:

\begin{thm}\label{teo6}
  Let $X$ and $Y$ be Banach spaces. Then the following hold.
  \begin{itemize}
  \item [(i)] If $X$ or $Y$ has the local diameter 2 property, then
    the projective tensor product $X \widehat{\otimes}_\pi Y$ has the
    local diameter 2 property.
  \item [(ii)] If $X$ or $Y$ has the diameter 2 property, then
    $X\oplus_\infty Y$ has the diameter 2 property.
  \item [(iii)] If $X$ and $Y$ both have the strong diameter 2
    property, then $X\oplus_1 Y$ has the strong diameter 2 property.
  \end{itemize}
\end{thm}

Here (i) follows by using that
$(X \widehat{\otimes}_\pi Y)^\ast
=\mathcal{L}(X,\Ys)=\mathcal{L}(Y,\Xs)$.
(ii) is Lemma 2.1 of \cite{LP}. The
proof of (iii) is the proof of Lemma 2.1 (ii) in \cite{BGLP}. 

\begin{rem}
  Note that the statement and proof of Lemma 2.1 (ii) in \cite{BGLP}
  do not match. What they actually prove is statement (iii) in Theorem
  \ref{teo6} above. We will show in Theorem~\ref{thm:stable-ell_p}
  that their statement is also true.
\end{rem}

We will end this section by mentioning that the interpolation spaces
$L_1(\mathbb {R}^{+})+L_\infty(\mathbb {R}^{+})$ (endowed with their
two natural norms) and $L_1(\mathbb {R}^{+}) \cap L_\infty(\mathbb
{R}^{+})$ (endowed with the maximum norm) all have the diameter $2$
property, but they do not have the Daugavet property \cite{AK}.

\section{Some examples not covered by the known classes of diameter 2 spaces}

In \cite{LP} it is proved that if a space has the diameter $2$
property then also the $\ell_\infty$-sum with any other space has this
property. Also, as pointed out in Theorem \ref{teo6}, the authors of
\cite{BGLP} proved that the $\ell_1$-sum of two spaces with the strong
diameter 2 property inherits this property. We will now prove that the
(local) diameter 2 property is in fact stable by taking $\ell_p$-sums
for every $1 \le p \le \infty$. First we need a lemma.

\begin{lem}\label{lem:diam2-scale}
  Suppose every non-empty relatively weakly open subset (resp. every
  slice, every finite convex combination of slices) of the unit ball
  $B_X$ of a Banach space $X$ has diameter $2$. Then every non-empty
  relatively weakly open subset (resp. every slice, every finite
  convex combination of slices) of $\delta B_X$ has diameter $2\delta$,
  where $1 \ge \delta > 0$.    
\end{lem}

\begin{proof}
  We first prove the result for non-empty relatively weakly open
  subsets. To this end, let $\varepsilon > 0$, $W \subset X$ be weakly open,
  and suppose $W \cap \delta B_X \not=\emptyset$. Pick some $x_0 \in
  B_X$ such that $\delta x_0 \in W \cap \delta B_X$ and some weak
  neighborhood $W_0=\{x \in X: |\xs_i(x-\delta x_0)| < \varepsilon, 1 \le i
  \le n\}$ of $\delta x_0$ with $W_0 \subset W$. Using the assumption
  there are $x_1,x_2 \in W_0'=\{x \in X:|\xs_i(x-x_0)| <
  \delta^{-1}\varepsilon, 1 \le i \le n\} \cap B_X$ such that $\|x_1-x_2\| >
  2-\varepsilon$. Now $\delta x_1, \delta x_2$ are in $W_0 \cap \delta B_X$
  and $\|\delta x_1-\delta x_2\| > (2-\varepsilon)\delta$. Since $\varepsilon$ was
  arbitrary we are done.

  The same argument works also for slices. 
  
  We now prove the result for finite convex combinations of slices. So
  let $\varepsilon > 0$, $S=\sum_{i=1}^n \lambda_i S_i$ a convex combination of
  slices of $\delta B_X$, and write $S_i=\{x \in \delta B_X: \xs_i(x) >
  \delta - \varepsilon_i\}$ with $\xs_i \in S_{\Xs}$. Let $S_i'=\{x \in B_X:
  \xs_i(x) > 1 - \delta^{-1}\varepsilon_i\}$. By assumption there are $x_1,x_2
  \in \sum_{i=1}^n \lambda_i S_i'$ with $\|x_1-x_2\| > 2-\varepsilon$. It is
  easy to see that $\delta x_1$ and $\delta x_2$ are in $S$ and
  $\|\delta x_1-\delta x_2\| > (2-\varepsilon)\delta$. Since $\varepsilon$ was
  arbitrary we are done.   
\end{proof}

\begin{thm}\label{thm:stable-ell_p}
  The (local) diameter 2 property is stable by taking $\ell_p$-sums
  for all $1 \le p \le \infty$.
\end{thm}

\begin{proof}
  We will only prove it for the diameter 2 property and for the case
  when we sum two spaces. The general case is similar. Also the proof
  for the local diameter 2 property is similar.

  To begin with note that the case $p=\infty$ is Theorem \ref{teo6}
  (ii). So let $1 \le p < \infty$, $\varepsilon > 0$, $X_1$ and $X_2$ Banach
  spaces with the diameter 2 property, and put $Z=X_1 \oplus_p
  X_2$. Let $W$ be a non-empty relatively weakly open subset of
  $B_Z$. Since $Z$ is infinite-dimensional every weakly open set is
  unbounded in some direction. Thus there is some $z_0 = (x_1,x_2) \in W
  \cap S_Z$. Now find some relatively weakly open neighborhood $W_0=\{z
  \in B_Z: |\zs_i(z-z_0)| < \varepsilon, i=1,2,\ldots,n \}$ of $z_0$ where
  $\zs_i = (\xs_{1,i},\xs_{2,i}) \in S_{\Zs}$ such that $W_0 \subset
  W$. Put 
  \begin{align*}
   W_0^1 &= \{x \in \|x_1\|B_{X_1}: |\xs_{1,i}(x-x_1)| < \varepsilon/2,
   i=1,2,\ldots,n\},\\ 
   W_0^2 &= \{x \in \|x_2\|B_{X_2}: |\xs_{2,i}(x-x_2)| < \varepsilon/2, i=1,2,\ldots,n\}
 \end{align*}

 Since, for $k=1,2$, $W_0^k$ is non-empty relatively weakly open in
 $\|x_k\|B_{X_k}$, we can by Lemma \ref{lem:diam2-scale} find points
 $x_k^1$ and $x_k^2$ in $W_0^k$ with $\|x_k^1-x_k^2\| >
 (2-\varepsilon)\|x_k\|$. Moreover, $W_0^\times = W_0^1 \times W_0^2 \subset
 W_0$. Indeed, suppose $y=(y_1,y_2) \in  W_0^\times$. Then
 \begin{equation*}
   \|y\|^p = \|y_1\|^p + \|y_2\|^p \le \|x_1\|^p + \|x_2\|^p = 1,
 \end{equation*}
 so $y \in B_Z$. Also for every $i=1,2,\ldots,n$ we have 
 \begin{align*}
   |\zs_i(y-z_0)|=|\xs_{1,i}(y_1-x_1) + \xs_{2,i}(y_2-x_2)| < \varepsilon,
 \end{align*}
 so $y \in W_0^\times$. Now put $x^1 = (x_1^1,x_2^1)$ and $x^2 =
 (x_1^2,x_2^2)$. Then we get
 \begin{align*}
   \|x^1-x^2\|^p &= \|x_1^1-x_1^2\|^p + \|x_2^1-x_2^2\|^p\\
   &\ge (2-\varepsilon)^p(\|x_1\|^p + \|x_2\|^p)\\
   &=(2-\varepsilon)^p,  
 \end{align*}
 and since $\varepsilon$ is arbitrary, we are done.    
\end{proof}

Using Theorem \ref{thm:stable-ell_p} it is easy to produce an example
which falls outside the theorems in Section 2.

\begin{eks}\label{c_nulleks}
  The space $X=c_0\oplus_2 c_0$ is a strict
  u-ideal in its bidual and has the diameter 2 property. However,
  $\Xs$ is not an $L$-summand in its bidual, $\sup_n\dim
  Z(X^{(2n})=1$, $X$ is not a uniform algebra,
  and $X$ lacks the Daugavet property. 
\end{eks}

\begin{proof} 
  $X$ is a strict u-ideal in its bidual.
  This can be shown as in \cite[Example~(5)]{GKS} or it can be shown
  directly using \cite[Lemma~2.2]{GKS}.

  Since every dual of $X$ contains an $\ell_2$-sum, $\Xs$ is in
  particular not an L-summand in its bidual and none of the duals of $X$
  can contain any non-trivial $M$-ideal (\cite[p.~45]{HWW}). Thus
  $Z(X^{(2n})=1$ for all $n$ (\cite[p.~39]{HWW}), hence $\sup_n \dim
  Z(X^{(2n})=1$. $X$ does not have the Daugavet property since no
  separable space with the Daugavet property can have an unconditional
  basis (\cite[Corollary~2.7]{KSSW}). Further, $X$ is definitely not a
  uniform algebra. Thus $X$ is not contained in any of the main cases
  covered in Section 2.
\end{proof}

Theorem \ref{teo2} tells us that proper $M$-embedded spaces have the
diameter 2 property.
Strict u-ideals in their biduals share many of the
properties of $M$-embedded spaces (see \cite{GKS}, \cite{LL},
\cite{AN1}, and \cite{HWW}). However the next example shows that
spaces which are strict u-ideals in their biduals need not even have
the local diameter 2 property.  

\begin{eks}
  $X=\mathbb{R} \oplus_1 c_0$ is a strict u-ideal in $\Xss$ which
  does not have the local diameter 2 property.
\end{eks}

\begin{proof}
  First, $X$ is a strict u-ideal by \cite[Example~(6)]{GKS}.
  To finish, let $\phi=(1,(0)) \in \Xs$ and
  $S=\{x \in B_X :\phi(x)>1-\varepsilon\}$.
  Then $S$ has diameter less than $2\varepsilon$.
\end{proof}

\section{Spaces with the strong diameter 2 property}\label{strongsection}
The object of this section is to prove that infinite-dimensional
uniform algebras, spaces
with the Daugavet property, proper $M$-embedded spaces, and biduals
of proper $M$-embedded spaces have
the strong diameter 2 property. We start by proving this for uniform
algebras. Let us first note the following simple, sufficient condition
for a space to have the strong diameter 2 property.
It will also be used in the proof of Theorem~\ref{thm-memb-strong}.

\begin{lem}\label{strongsufficient} Suppose $X$ has the local diameter
  2 property in the following sense: Whenever $\{S_{j}\}_{j=1}^n$ is
  a finite family of slices of $B_X$ and $\varepsilon>0$, then there exist
  points $h_j\in S_{j}$ and $\varphi\in B_X$, independent of $j$, such
  that $h_j\pm\varphi\in S_{j}$ and $\|\varphi\|
  >1-\varepsilon$. Then $X$ has the strong diameter 2 property.
\end{lem}

\begin{proof} Let $S$ be a finite convex combination of slices, i.e.,
  $S=\sum_{j=1}^n \lambda_j S_{j}$ where $0<\lambda_j<1$ and
  $\sum_{j=1}^n \lambda_j=1$. For every $1\leq j\leq n$, take $h_j\in
  S_j$ and $\varphi\in B_X$ such that $h_j\pm\varphi\in S_j$ and
  $\|\varphi\|>1-\varepsilon$. Define
  \begin{equation*}
    \psi_+=\sum_{j=1}^n \lambda_j h_j + \varphi
    \:\:\mbox{and}\:\:
    \psi_-=\sum_{j=1}^n \lambda_j h_j - \varphi.
  \end{equation*}
  Then
  $\psi_+,\psi_-\in S$ and $\|\psi_+ -\psi_-\|=\|2\varphi\|>2-2\varepsilon$.
\end{proof}

\begin{thm}\label{cor:unifalg-strong}
  Infinite-dimensional uniform algebras have the strong diameter 2 property.
\end{thm}

\begin{proof}
  An inspection of the proof of \cite[Theorem 2]{NW}
  shows that
  a uniform algebra fulfills the
  conditions in Lemma~\ref{strongsufficient}.
  Using their notation, let $\varepsilon = \min_j \varepsilon_j$ and
  choose
  $0<\delta\leq \varepsilon/12$ instead
  (to simplify for point (ii) below) and write
  \begin{equation*}
    \frac{h_j}{1+4\delta}\pm\frac{\varphi}{1+4\delta}
    =  h'_{j}\pm\varphi'.
  \end{equation*}
  Then $h'_{j}\pm\varphi' \in S_j(\ell_j,\varepsilon_j)$,
  since by using facts from the proof
  of \cite[Theorem 1]{NW}, we get
  \begin{itemize}
  \item [(i)] $\|h'_{j}\|\leq (1+3\delta)/(1+4\delta)<1$,
  \item [(ii)] $1 > \|\varphi'\|=\frac{1}{1+4\delta}
    >\frac{1}{1+\frac{\varepsilon}{3}}>1-\varepsilon$,
  \item [(iii)]
    $\|h_j' \pm \varphi'\| \le 1$,
  \item [(iv)] $\ell_j (h_{j})\geq 1-\delta -4\delta
    =1-5\delta$, and
  \item [(v)] $|\ell_j(\varphi)| < 2\delta$.
  \end{itemize}
  Finally (iv) and (v) imply that
  \begin{equation*}
    \ell_j(h'_j \pm \varphi')
    = \frac{\ell_j(h_j)-|\ell_j(\varphi)|}{1+4\delta}
    > 
    \frac{1-5\delta-2\delta}{1+4\delta}
    \ge
    \frac{1-\frac{7\varepsilon}{12}}{1+\frac{\varepsilon}{3}}
    \ge 1-\varepsilon_j.
  \end{equation*}
\end{proof}

For the proof that spaces with the Daugavet property
have the strong diameter 2 property we will need a
version of \cite[Lemma~2.1]{KSSW}.
The lemma is used in the proof of \cite[Lemma~3]{SHV},
but is not stated explicitly. We include a proof for easy reference.

\begin{lem}\label{lem:shvydkoylem2.2}
  If $X$ has the Daugavet property, then
  for every $\varepsilon > 0$, every $y_0 \in X$ and every slice $S(x_0^*,\alpha_0)$
  of $B_X$ there is another slice $S(x_1^*,\alpha_1) \subset S(x_0^*,\alpha_0)$
  such that for every $x \in S(x_1^*,\alpha_1)$
  the inequality $\|\lambda x + y_0\| \ge
  \lambda + \|y_0\| - \varepsilon$
  holds for every $0 \le \lambda \le 1$.
\end{lem}

\begin{proof}
  Only the case $y_0 \neq 0$ needs proof. By making the slice
  $S(x_0^*,\alpha_0)$ smaller we may assume
  without loss of generality
  that $2\alpha_0 \|y_0\| < \varepsilon$.

  Define $T = x_0^* \otimes y_0$, so that $\|T\| = \|y_0\|$.
  By the Daugavet equation $\|I^* + T^*\| = \|I+T\| = 1 + \|y_0\|$.
  In order to guarantee that $0<\alpha_1<1$, one can choose $y^* \in
  S_{\Xs}$ such that $\|(I+T)^*y^*\| \ge 1 + \|y_0\|(1-\alpha_0)$
  and $y^*(y_0) \ge 0$.
  Define
  \begin{equation*}
    x_1^* = \frac{(I+T)^*y^*}{\|(I + T)^*y^*\|}
    \qquad \mbox{and} \qquad
    \alpha_1 =
    1 - \frac{1 + \|y_0\|(1 - \alpha_0)}{\|(I + T)^*y^*\|}.
  \end{equation*}
  Then, for all $x \in S(x_1^*,\alpha_1)$,
  \begin{equation*}
    y^*(x) + y^*(y_0)x_0^*(x) =
    \langle (I+T)^* y^*, x \rangle
    \ge (1-\alpha_1)\|(I+T)^* y^*\|
    = 1 + \|y_0\|(1 - \alpha_0).
  \end{equation*}
  We get
  \begin{equation}\label{eq:1}
    \|y_0\|x_0^*(x) \ge y^*(y_0)x_0^*(x)
    \ge 1 + \|y_0\|(1-\alpha_0) - y^*(x)
    \ge \|y_0\|(1-\alpha_0),
  \end{equation}
  hence $x_0^*(x) \ge 1-\alpha_0$
  which shows that $S(x_1^*,\alpha_1) \subseteq S(x_0^*,\alpha_0)$.
  Since $x_0^*(x) \le 1$ we get
  \begin{equation}\label{eq:2}
    y^*(x) + y^*(y_0) \ge y^*(x) + y^*(y_0)x_0^*(x)
    \ge 1 + \|y_0\|(1-\alpha_0),
  \end{equation}
  and thus $\|x + y_0\| \ge 1 + \|y_0\|(1-\alpha_0)$.
  
  Finally, by \eqref{eq:1} and \eqref{eq:2},
  \begin{align*}
    y^*(\lambda x) + y^*(y_0)
    &= \lambda y^*(x) + y^*(y_0) \\
    &\ge \lambda \bigl(1 + \|y_0\|(1-\alpha_0) - y^*(y_0)x_0^*(x)\bigr)
    + \|y_0\|(1-\alpha_0) \\
    &\ge \lambda (1 + \|y_0\|(1-\alpha_0) - \|y_0\|)
    + \|y_0\|(1-\alpha_0) \\
    &= \lambda + \|y_0\| - \alpha_0 \|y_0\|(1+\lambda).
  \end{align*}
  Hence $\|\lambda x + y_0\| \ge \lambda + \|y_0\| -
  \alpha_0 \|y_0\|(1+\lambda) > \lambda + \|y_0\| - \varepsilon$.
\end{proof}

Now we are ready to show that spaces with the Daugavet property have
the strong diameter $2$ property.
The idea of the proof is due to Shvydkoy (see \cite[Lemma~3]{SHV}),
but we have to apply this idea twice.

\begin{thm}\label{thm:daugavet-strong}
  If $X$ has the Daugavet property, then
  $X$ has the strong diameter $2$ property.
\end{thm}

\begin{proof}
  Let $S_j = \{x \in B_X : x_j^*(x) > 1 - \varepsilon_j\}$
  be slices of $B_X$ and let $0 < \lambda_j < 1$
  such that $\sum_{j=1}^n \lambda_j = 1$.
  
  Let $0<\varepsilon <1$ and
  $y \in S_X$. Using Lemma~\ref{lem:shvydkoylem2.2}
  we can find $x_1 \in S_1$ such that
  $\|\lambda_1 x_1 + y\| > \lambda_1 + 1+\varepsilon/n$.
  Using Lemma~\ref{lem:shvydkoylem2.2} repeatedly
  we find $x_j \in S_j$ such that
  \begin{equation*}
    \|\sum_{j=1}^n \lambda_j x_j + y\| > \sum_{j=1}^n \lambda_j  + 1
    - \varepsilon = 2 - \varepsilon.
  \end{equation*}
  In particular,
  \begin{equation*}
    1 \ge \| \sum_{j=1}^n \lambda_j x_j \|
    \ge \| \sum_{j=1}^n \lambda_j x_j  + y\| - \|-y\|
    \ge 2 - \varepsilon - 1 = 1-\varepsilon.
  \end{equation*}
  Define $y_0 = \frac{\sum_{j=1}^n \lambda_j x_j}{\|\sum_{j=1}^n
    \lambda_j x_j\|}$.
  Then $\|y_0\| = 1$ and $\|y_0 - \sum_{j=1}^n \lambda_j x_j\| \le
  \varepsilon$.
  
  Repeat the procedure above using $-y_0$ instead of $y$
  and find $z_j \in S_j$ such that
  \begin{equation*}
    \|\sum_{j=1}^n \lambda_j z_j - y_0\| \ge 2 - \varepsilon.
  \end{equation*}
  We get
  \begin{equation*}
    \|\sum_{j=1}^n \lambda_j z_j - \sum_{j=1}^n \lambda_j x_j\|
    \ge \|\sum_{j=1}^n \lambda_j z_j - y_0\|
    - \|y_0 - \sum_{j=1}^n \lambda_j x_j\|
    \ge 2 - 2\varepsilon.
  \end{equation*}
  This shows the existence of points in the convex combination
  of the slices with distance arbitrarily close to $2$.
\end{proof}

We will end this section by showing that the bidual of proper $M$-embedded
spaces have the strong diameter 2 property. To do this we will need
some results inspired by \cite{LP}. 
The first lemma is contained in \cite[Lemma~2.1]{LP},
but we provide a short elementary proof.

\begin{lem}\label{lem:vegard}
  Let $X$ and $Y$ be Banach spaces, $W$ a weakly open subset in $Z=X
  \oplus_\infty Y$, and $(x_0,y_0) \in W$. Then there exist weakly open
  subsets $U$ of $X$ and $V$ of $Y$ such that $(x_0,y_0) \in U
  \times V \subset W$. Moreover, if $W$ is a relatively weakly
  open subset of $B_Z$, then $U$ and $V$ can be chosen to be relatively
  weakly open subsets of $B_X$ and $B_Y$ respectively.
\end{lem}

\begin{proof}
  There exists $f_i =(\xs_i,\ys_i) \in \Xs \times \Ys$ where
  $i=1,2,\ldots, n$ such that
  \begin{equation*}
    W_0 = \{(x,y) \in Z: |f_i(x,y)-f_i(x_0,y_0)| < 1,
    i=1,2, \ldots, n\}
    \subset W.
  \end{equation*}
  Let $U_0 = \{x \in X: |\xs_i(x)-\xs_i(x_0)| < \frac{1}{2}, i=1,2,\ldots, n\}$
  and
  $V_0 = \{y \in Y: |\ys_i(y)-\ys_i(y_0)| < \frac{1}{2}, i=1,2,\ldots, n\}$.
  Then $U_0$ and $V_0$ are weakly open in
  $X$ and $Y$ respectively and $U_0 \times V_0 \subset W_0 \subset
  W$.

  For the last part, just write $U = U_0 \cap B_X$ and $V = V_0 \cap B_Y$.
\end{proof}

Lemma~2.1 in \cite{LP} shows that an $\ell_\infty$-sum of two Banach spaces
has the diameter $2$ property if one of the components has.
Next we show that a similar result holds for the strong diameter $2$ property.

\begin{prop}\label{prop:vegard2}
  Let $X$ and $Y$ be Banach spaces and suppose $X$ has the strong
  diameter 2 property. Then $X \oplus_\infty Y$ has the strong
  diameter 2 property.
\end{prop}

\begin{proof}
  Let $Z=X \oplus_\infty Y$ and let $P:Z \to X$ be the natural projection onto
  $X$.
  Let $W=\sum_{i=1}^n \lambda_i S_i$ be a convex combination of
  slices
  $S_i = \{(x,y) \in B_Z: g_i(x,y) > 1-\varepsilon_i\}$ for some $g_i \in
  S_{Z^{\ast}}$ and $\varepsilon_i > 0$.
  Obviously $W \subset B_Z$.
  It is enough to show that $P(W)$
  is non-empty and contains a non-empty finite convex combination of
  relatively weakly open subsets of $B_X$.
  Indeed, the conclusion then follows from Bourgain's
  result (see \cite[Lemma II.1 p. 26]{GGMS}) 
  that every relatively weakly open subset of the unit ball
  contains a convex combination of slices and from the
  assumption, since $\|P\|=1$.

  Since $W$ is
  non-empty, there is some $(x_0,y_0) \in W$. Thus $x_0=P(x_0,y_0) \in
  P(W)$, so $P(W)$ is non-empty.
  Finally, write $(x_0,y_0)=\sum_{i=1}^n \lambda_i (x_0^i,y_0^i)$ where
  $(x_0^i,y_0^i) \in S_i$. Since each slice $S_i$ is relatively weakly
  open, it is possible by Lemma \ref{lem:vegard} to find relatively
  weakly open sets $U_i \subset B_X$ and $V_i \subset B_Y$ such that
  $(x_0^i,y_0^i) \in U_i \times V_i \subset S_i$. Thus
  \begin{multline*}
    P(W) \supset P\biggl(\sum_{i=1}^n \lambda_i (U_i \times V_i) \cap B_Z\biggr)
    \\ =
    P\biggl((\sum_{i=1}^n \lambda_i U_i) \times (\sum_{i=1}^n \lambda_i
    V_i) \cap \bigl(B_X \times B_Y\bigr)\biggr) = \sum_{i=1}^n \lambda_i U_i \cap
    B_X.
  \end{multline*}
  Since each $U_i$ is weakly open, we are done.
\end{proof}

Using the above proposition we can strengthen
\cite[Proposition~2.6]{LP}.

\begin{prop}\label{prop2.6iLP-strongversjon}
  Let $X$ be a Banach space containing a subspace isomorphic
  to $c_0$. Then there exists an equivalent norm on $X$ such that
  $X$ has the strong diameter $2$ property in the new norm.
\end{prop}

\begin{proof}
  The proof is similar to the proof of \cite[Proposition~2.6]{LP}.
  The only difference is that we use that $\ell_\infty$ has the
  strong diameter $2$ property and Proposition~\ref{prop:vegard2}.
\end{proof}

Next we have dual versions of Lemma~\ref{lem:vegard}
and Proposition~\ref{prop:vegard2}.

\begin{lem}\label{lem:vegard*}
  Let $X$ be a Banach space and $Y$ a closed subspace of $X$. Assume
  that $\Xss = Y^{\perp\perp} \oplus_\infty Z^\perp$ for some closed
  subspace $Z$ of $\Xs$. Let $W$ be a weak$^\ast$ open subset in $\Xss$
  and $(\xss_0,\yss_0) \in W$. Then there exist weak$^\ast$ open subsets
  $U$ of $Y^{\perp\perp}$ and $V$ of $Z^{\perp}$ such that
  $(\xss_0,\yss_0) \in U \times V \subset W$. Moreover, if $W$ is
  a relatively weak$^\ast$ open subset of $B_{\Xss}$, then $U$ and $V$
  can be chosen to be relatively weak$^\ast$ open subsets of
  $B_{Y^{\perp\perp}}$ and $B_{Z^{\perp}}$ respectively.
\end{lem}

\begin{proof}
  The proof is similar to the proof of Lemma \ref{lem:vegard}.
\end{proof}

\begin{prop}\label{prop:vegard*}
  Let $Y$ be a closed subspace of a Banach space $X$ such that $\Xss =
  Y^{\perp\perp} \oplus_\infty Z^{\perp}$ for some closed subspace $Z$
  of $\Xs$. If every finite convex combination of weak$^\ast$
  slices of $Y^{\perp\perp}$ has diameter 2 then every finite
  convex combination of weak$^\ast$ slices of $\Xss$ has diameter 2.
\end{prop}

\begin{proof}
  The proof is similar to the proof of Proposition \ref{prop:vegard2}.
\end{proof}

We conclude this section by
proving that if $Y$ is a proper $M$-ideal in $X$ and the range of the
$L$-projection in $\Xs$ is 1-norming for $X$, then both $X$ and $Y$
have the strong diameter 2 property.
The proof is inspired by the proof of \cite[Proposition~2.3]{LP}.
\begin{thm}\label{thm-memb-strong}
  Let $Y$ be a proper $M$-ideal in $X$, i.e. $\Xs = Z \oplus_1 Y^{\perp}$
  for some subspace $Z$ of $\Xs$. If $Z$ is 1-norming for $X$, then
  both $X$ and $Y$ have the strong diameter 2 property.
  
  In particular, if $X$ is proper $M$-embedded, then both $X$ and $\Xss$
  have the strong diameter 2 property.
\end{thm}

\begin{proof}
  We have $\Xss = Y^{\perp\perp} \oplus_\infty Z^\perp$.
  Let $i=1,2,\ldots,n$, $\varepsilon_i>0$, and $z_i \in S_Z$, and
  consider the weak$^\ast$ slices
  \begin{equation*}
    S_i^\ast=\{y^{\perp\perp} \in B_{Y^{\perp\perp}}:
    y^{\perp\perp}(z)>1-\varepsilon_i\}.
  \end{equation*}
  Let $\lambda_i > 0$ such that $\sum_{i=1}^n \lambda_i = 1$.
  First we prove that the convex combination $S^\ast=\sum_{i=1}^n
  \lambda_i S_i^\ast$ has diameter 2. Note that by
  Goldstine's theorem the slice $S_i = S(z_i,\varepsilon_i)$ of $B_Y$ is weak$^\ast$
  dense in $S_i^\ast$ for each $i$. In particular $S_i^\ast \cap B_Y \neq
  \emptyset$ for every $i=1,2,\ldots,n$.
  Choose $y_0^i \in S_i^\ast \cap B_Y$.
  
  Given $\delta>0$ it is possible to find $x \in S_X$ such that $\|x
  +Y\|>1-\delta$ since $Y$ is a proper subspace of $X$.
  Using the $M$-ideal property, by \cite[Proposition 2.3]{W2} there
  is a net $(y_d)$ in $Y$ such that $y_d \to x$ in the
  $\sigma(X,Z)$-topology with
  $\limsup_d\|y_0^i \pm (x-y_d)\| \le 1$. Thus for any given
  $0<r_i<1$ we can find $d_0^i$ such that
  \begin{equation*}
    r_i\|y_0^i \pm (x-y_d)\| \le 1
  \end{equation*}
  whenever $d \ge d_0^i$.
  Since each $S_i$ is weakly open in $Y$ and the
  $\sigma(X,Z)$-topology on $Y$ is just the weak topology
  we may assume that
  \begin{equation*}
    r_i(y_0^i \pm (x-y_d)) \in S_i \subset S_i^\ast
  \end{equation*}
  for $d \ge d_0^i$. Let $r=\max_{i} r_i$ and $d_0 \ge d_0^i$ for
  every $i=1,2,\ldots,n$.
  Using Lemma~\ref{strongsufficient} with $h_i = r y_0^i$
  and $\varphi = r(x-y_d)$ we get that $S^\ast$ has diameter 2
  since
  \begin{equation*}
    \|\varphi\| \ge r\|x+Y\| > r(1-\delta),
  \end{equation*}
  and $r$ and $\delta$ can be chosen arbitrarily close to $1$
  and $0$ respectively.
  
  From Proposition~\ref{prop:vegard*} and
  $\Xss = Y^{\perp\perp} \oplus_\infty Z^\perp$
  we get that every finite convex combination of weak$^\ast$
  slices of $B_\Xss$ has diameter 2. By the weak$^\ast$
  density of $B_X$ in $B_{\Xss}$ and the weak$^\ast$ lower
  semi-continuity of the norm, it follows that $X$ has
  the strong diameter 2 property.

  Since $Z$ is 1-norming the norm on $X$ is
  $\sigma(X,Z)$-lower semi-continuous.
  That $Y$ has the strong diameter 2 property is
  immediate since a functional $\ys \in S_{\Ys}$
  uniquely extends to a functional in $S_{\Xs}$
  and the slice $S(\ys,\varepsilon)$ of $B_Y$ is $\sigma(X,Z)$-dense
  in the slice $S(\ys,\varepsilon)$ of $B_X$.
\end{proof}

\section{Some concluding remarks and questions}\label{questsection}
As remarked in Section 1 we do not know if the three diameter 2
properties really are different. Having an answer to this question is
clearly important for future research on diameter 2 spaces. Our
conjecture is that they are not equal.

Meanwhile, and especially if our conjecture is correct, some questions
naturally come to mind:

\begin{itemize}
\item [(a)] Can one conclude diameter 2 property or even strong
  diameter 2 property in Proposition 2.2?
\item [(b)] From Theorems 2.6 and 2.7(i), how are diameter 2
  properties in general preserved by tensor products? (An important
  recent contribution here is \cite{ABR}.)  
\item [(c)] Is Theorem 3.2 true for the strong diameter 2 property?
\item [(d)] Does $\dim Z(X^{(\infty})=\infty$ imply strong diameter 2 property?
\item [(e)] Note that $X$ inherits all of the three diameter 2
  properties from $X^{\ast\ast}$. In general, which subspaces of a
  space with the (local, strong) diameter 2 property inherits this
  property? 
\end{itemize}

\section*{Acknowledgement}
The authors would like to thank the referee for many valuable
suggestions that helped to improve the presentation.

\providecommand{\bysame}{\leavevmode\hbox to3em{\hrulefill}\thinspace}
\providecommand{\MR}{\relax\ifhmode\unskip\space\fi MR }
\providecommand{\MRhref}[2]{%
  \href{http://www.ams.org/mathscinet-getitem?mr=#1}{#2}
}
\providecommand{\href}[2]{#2}

\end{document}